\documentclass[11pt]{article}

\usepackage{amsmath, amsthm, amsfonts}
\usepackage[bottom]{footmisc}
\usepackage{geometry}
\newtheorem{lemma}{Lemma}[section]
\newtheorem{thm}{Theorem}[section]
\newtheorem{prop}[lemma]{Proposition}
\newtheorem{cor}[lemma]{Corollary}

\theoremstyle{definition}
\newtheorem{exe}{Example}[section]
\usepackage[LY1]{fontenc}
\usepackage{graphicx}
\usepackage{float}
\usepackage{dsfont}
\usepackage{amssymb}
\usepackage{amsmath}
\usepackage{times}
\usepackage{csquotes}
\usepackage{yfonts}
\usepackage{bbm}
\usepackage{dsfont}
\usepackage{enumitem}
\usepackage{url}

\numberwithin{equation}{section}

\title{Contradictory predictions with multiple agents }
\author{Stanis\l{}aw Cichomski\footnote{Faculty of Mathematics, Informatics and Mechanics, University of Warsaw (Poland). 
E-mail: s.cichomski@uw.edu.pl}  , \ Adam Os{\k e}kowski\footnote{Faculty of Mathematics, Informatics and Mechanics, University of Warsaw (Poland).  E-mail: A.Osekowski@mimuw.edu.pl}}

\begin{document}
\maketitle

\begin{abstract}
Let $X_1$, $X_2$, $\ldots$, $X_n$ be a sequence of coherent random variables, i.e., satisfying the equalities
$$ X_j=\mathbb{P}(A|\mathcal{G}_j),\qquad j=1,\,2,\,\ldots,\,n,$$
almost surely for some event $A$. The paper contains the proof of the estimate
$$\mathbb{P}\Big(\max_{1\le i < j\le n}|X_i-X_j|\ge \delta\Big) \leq  \frac{n(1-\delta)}{2-\delta} \wedge 1,$$
where $\delta\in (\frac{1}{2},1]$ is a given parameter. The inequality is sharp: for any $\delta$, the constant on the right cannot be replaced by any smaller number.  The argument rests on several novel combinatorial and symmetrization arguments, combined with dynamic programming. Our result generalizes the two-variate inequality of K. Burdzy and S. Pal  and in particular provides its alternative derivation.
 \end{abstract}

 \section{Introduction}
 
 Let $n$ be a positive integer. Following \cite{C1}, we say that a sequence $(X_1,\,X_2,\, \dots,\, X_n)$ of random variables on a given probability space $(\Omega,\mathcal{F},\mathbb{P})$ is  coherent, if there is a  sequence of sub-$\sigma$-fields $\mathcal{G}_1,\,\mathcal{G}_2,\,\dots,\, \mathcal{G}_n$ of $\mathcal{F}$ and an event $A \in \mathcal{F}$ such that
\begin{equation} \label{Cdef}X_j \ = \ \mathbb{P}(A|\mathcal{G}_j), \ \ \ \ \ j=1,\,2,\,\dots,\,n. \end{equation}
In such a case,  we write $(X_1, X_2,\dots ,X_n) \in \mathcal{C}$ and the joint distribution of the vector $(X_1,\,X_2,\,\ldots,\,X_n)$ is also  said to be coherent. This setup has a nice and transparent interpretation, which is important for many applications. Namely,  suppose that a group of $n$ experts provides their personal estimates on the likelihood of some random event $A$, and assume that the knowledge of $j$-th expert is represented by the $\sigma$-algebra $\mathcal{G}_j$, $j=1,\,2,\,\ldots,\,n$. Then the predictions $X_1$, $X_2$, $\ldots$, $X_n$ of the experts are given by \eqref{Cdef}. In general, there are three basic categories of problems which are studied in the above context, stemming from applications in statistics, decision theory, economics, game theory as well as probability and information theory. 

\smallskip

$\cdot$ \textit{Optimal combining} -- depending on purpose, to find an optimal procedure that combines  multiple coherent opinions in order to produce a better forecast; see \cite{C1, C2, C4, C3}.

\smallskip

$\cdot$ \textit{Bayesian persuasion} --  to compute (given a specific payoff function) how much one of the agents can benefit by selectively revealing parts of his information to other players, thus changing their beliefs and reactions; see \cite{B4, B2, B1, B3}.

\smallskip

$\cdot$ \textit{Maximal discrepancy} -- to provide sharp bounds on the maximal possible spread of coherent opinions. For instance, for a fixed functional $\Phi: [0,1]^n \rightarrow \mathbb{R}_+$, evaluate
$$\sup \mathbb{E}\Phi(X_1,X_2,\dots, X_n),$$
where the supremum is taken over all probability models as described above; see \cite{contra, pitman, EJP, BPC}.

\smallskip

The contribution of this paper concerns the last category. Our motivation comes from the following foundational result of K. Burdzy and S. Pal \cite{contra}. 

\begin{thm} \label{contr} For any  threshold $\delta \in (\frac{1}{2},1]$, we have
\begin{equation} \label{Eq2}\sup_{(X,Y)\in \mathcal{C}} \mathbb{P}(|X-Y|\ge \delta) = \frac{2(1-\delta)}{2-\delta}. \end{equation} \end{thm}
 
In the language of applications, Theorem \ref{contr}  establishes a sharp upper bound for the probability that two experts, with access to different information sources, will deliver highly incongruent (or contradictory)  opinions.  The original proof of equality (\ref{Eq2}) is remarkably complex and rather difficult: an explicit optimizer is obtained by a series of consecutive reductions and simplifications. As pointed out in \cite{pitman}, finding a simpler proof of this result would be highly desirable. Another natural and important question concerns the extension of the threshold bound (\ref{Eq2}) to $n > 2$ coherent opinions. Our main result in this paper is as follows, we use the notation $a\wedge b$ for the minimum of the numbers $a$ and $b$.

\begin{thm} \label{contrN} For any  threshold $\delta \in (\frac{1}{2},1]$ and  every integer $n\ge 2$, we have
\begin{equation} \label{EqN} \sup_{(X_1,X_2,\dots, X_n)\in \mathcal{C}} \mathbb{P}\Big(\max_{1\le i < j\le n}|X_i-X_j|\ge \delta\Big) = \frac{n(1-\delta)}{2-\delta} \wedge 1.\end{equation} \end{thm} 

Correspondingly, Theorem \ref{contrN} expands the range of applications from two experts to multiple agents scenario. Quite unexpectedly (at least to the authors), the threshold bound (\ref{EqN}) reveals an almost linear dependence between the examined quantities and the number of coherent random variables. The proof of (\ref{EqN}) that we present below is completely independent from the reasoning in \cite{contra} and hence can be regarded as an alternative demonstration of (\ref{Eq2}). Moreover, our approach does not refer in any significant way to the particular choice of integer $n$. 

The above statement should also be compared to its version concerning the maximal spread of expectations. For any pair $(X,Y)$ of coherent random variables we have the sharp estimate
$$ \mathbb E |X-Y|\leq \frac{1}{2}$$
(see e.g. \cite{pitman}). The paper \cite{EJP} contains the extension of this result to the case of an arbitrary number of variables.

\begin{thm}
Under the above notation, we have
\begin{equation}\label{obj}
 \sup_{(X_1,X_2,\dots, X_n)\in \mathcal{C}} \  \mathbb{E}  \max_{1\le i<j\le n}  |X_i-X_j|=\begin{cases}
\frac{1}{2} & \mbox{if }n=2,\smallskip\\
2-\sqrt{2} & \mbox{if }n=3,\smallskip\\
\frac{7}{2}-2\sqrt{2} & \mbox{if }n=4,\smallskip\\
\displaystyle \frac{n-2}{n-1} & \mbox{if }n\geq 5.
\end{cases}
\end{equation}
\end{thm}

It might be a little unexpected that the above upper bound involves four different formulas depending on the value of $n$, while \eqref{EqN} is formulated with a single expression. We would also like to mention that the linearity of the expectation makes the proof of \eqref{obj} a bit simpler; the successful treatment of the tail inequality will require a wider class of arguments.

Let us say a few words about our approach and the organization of the remaining part of the  paper. 
In the next section we  apply a special symmetrization technique, which reduces the problem of calculating the left-hand side of (\ref{EqN}) to the combinatorial  optimization over specific objects of geometrical nature. This approach seems to be especially beneficial due to certain convenient symmetry constraints it enforces. Then, in Section 3, using dynamic programming arguments, we solve the previously obtained optimization problem. This boils down to the derivation of a suitable Bellman function and some further reductions of the problem. This  appearance of dynamic programming is not  surprising: as evidenced in numerous papers, the Bellman function method is  a powerful tool used widely in martingale theory and harmonic analysis to obtain sharp inequalities -- see e.g. \cite{D2, D3, D1} and consult the references therein.  %What is perhaps a little unexpected in the context of \eqref{EqN} is the lack of any martingale (or dyadic) structure which typically indicates the possibility of the use of dynamic programming. 
We strongly believe that our argumentation can be pushed further and successfully applied in the further study of coherent distributions.

\section{Basic reductions and symmetrizations}

Throughout, we assume that $n\geq 2$ is a fixed integer and $\delta\in (\frac{1}{2},1]$ is a given threshold. 
We begin with the standard discretization, which will later allow us to pass to various combinatorial and optimization arguments. Let $\mathcal{C}(n,m)$ be the family of all  vectors $X = (X_1,\,X_2,\,\ldots,\,X_n) \in \mathcal{C}$ such that each $X_j$ takes at most $m$ different values, $j=1,\,2,\,\ldots,\,n$. 

\begin{prop}\label{DISCR} To prove the threshold bound (\ref{EqN}), it is enough to verify that
\begin{equation} \sup_{\substack{m  \in \{1,2,\dots\} \\ X \in\mathcal{C}(n,m)}} \mathbb{P}\Big(\max_{1\le i < j\le n}|X_i-X_j|\ge \delta\Big) \ =  \ \frac{n(1-\delta)}{2-\delta} \wedge 1. \label{EqNM} \end{equation}
\end{prop}
\begin{proof} Assume that (\ref{EqNM}) holds and fix any $n$-variate vector $X\in \mathcal{C}$. Let $m$ be a positive integer with $\delta > \frac{2}{m}+\frac{1}{2}$. As shown in \cite{contra, mastersthesis}, there exists a vector $X^{(m)}\in \mathcal{C}(n,m)$ such that $|X_j-X_j^{(m)}|\le \frac{1}{m}$ almost surely for all indices $j=1,\,2,\,\dots, \,n$.
Thus, by the triangle inequality, we have 
$$\mathbb{P}\Big(\max_{1\le i < j\le n}|X_i-X_j|\ge \delta\Big) \  \le  \ \mathbb{P}\Big(\max_{1\le i < j\le n}|X_i^{(m)}-X_j^{(m)}|\ge \delta-\frac{2}{m}  \Big) \ \le \ \frac{n\Big(1-(\delta-\frac{2}{m})\Big)}{2-(\delta-\frac{2}{m})},$$
where the second inequality follows from (\ref{EqNM}). Passing with $m$ to infinity ends the proof.
\end{proof}
 
Later on, we will need the following structural fact.
 
\begin{lemma}\label{lemat} 
Assume that $\{G_1,G_2, \dots, G_m\}$ is a finite partition of \ $\Omega$, let $A\in \mathcal{F}$ be an arbitrary event and put $Y=\mathbb{E}(\mathbbm{1}_A|\sigma(G_1,G_2,\ldots,G_m))$. Then for any $y\in (0,1]$ such that $\mathbb{P}(Y=y)>0$, we have
$$
\mathbb{P}(\{Y=y\}\cap A^c) \  =  \ \mathbb{P}(\{Y=y\}\cap A)\cdot \frac{1-y}{y}.
$$
\end{lemma}
\begin{proof} This is straightforward. For any $G\in \sigma(G_1,G_2,\ldots,G_m)$  such that $\mathbb{E}(\mathbbm{1}_A|G)=y$, we write
 $$y = \frac{\mathbb{P}(A\cap G)}{\mathbb{P}(G)}=\frac{\mathbb{P}(A\cap G)}{\mathbb{P}(A\cap G)+\mathbb{P}(A^c\cap G)}.$$
 This is equivalent to $y\cdot \Big(\mathbb{P}(A\cap G)+\mathbb{P}(A^c\cap G)\Big) = \mathbb{P}(A\cap G)$, or $\mathbb{P}(A^c\cap G)=\frac{1-y}{y}\cdot\mathbb{P}(A\cap G).$ It remains to take $G=\{Y=y\}$; we have $G\in \sigma(G_1,G_2,\ldots,G_m)$, since $Y$ is measurable with respect to the latter $\sigma$-algebra. 
\end{proof}
 
Now we will describe a useful symmetrization procedure, which  allow us to replace the left-hand side of \eqref{EqNM} with a more regular expression (see Corollary \ref{reduction} below). We need some additional notation. Fix a positive integer $m$ and let $X\in \mathcal{C}(n,m)$ be a coherent  vector with $X_i=\mathbb E(\mathbbm{1}_A|\mathcal{G}_i)$, $i=1,\,2,\,\ldots,\,n$. Let $U$ be a random variable independent of $\mathcal{G}_1,\,\mathcal{G}_2,\,\ldots,\,\mathcal{G}_n$ and $A$, having the two-point distribution $\mathbb{P}(U=0)=\mathbb{P}(U=1)=1/2$. Then  $\widetilde{X}$, the mixture of vectors $X$ and $1-X$, is given by
$$( \widetilde{X}_1, \widetilde{X}_2, \dots, \widetilde{X}_n) \ = \ U \cdot (X_1, X_2, \dots, X_n) + (1 - U) \cdot (1 - X_1, 1 - X_2, \dots, 1 - X_n).$$
Furthermore, we define the mixture $\widetilde{A}$ of $A$ and $A^c$ by the requirement $\mathbbm{1}_{\widetilde{A}}=\widetilde{\mathbbm{1}_A}$, or more explicitly, $\widetilde{A} = (A \cap \{U = 1\})\cup  (A^c \cap \{U = 0\})$. Let us distinguish the $\sigma$-algebras $\widetilde{\mathcal{G}}_i=\sigma(\mathcal{G}_i, U)$, $i=1,\,2,\,\ldots,\, n$. The key properties of these objects are summarized in a statement below.

\begin{prop} \label{SYM7} 
Under the above notation, the following holds true.

(i) We have $\mathbb{P}(\widetilde{A})=\frac{1}{2}$, $\widetilde{X}\in \mathcal{C}(n, 2m)$ and $\widetilde{X}_i=\mathbb{E}(\mathbbm{1}_{\widetilde{A}}|\mathcal{\widetilde{G}}_i)$ for all $i$.

(ii) For any sequence $(x_i)_{i=1}^n\subset [0,1]$ we have the identity
$$\mathbb{P}\Big(\bigcap_{i=1}^{n}\{\widetilde{X}_i=x_i\} \cap \widetilde{A} \Big) \ = \ \mathbb{P}\Big(\bigcap_{i=1}^{n}\{\widetilde{X}_i=1-x_i\} \cap \widetilde{A}^c\Big). $$

(iii) For any $x \in (0,1]$,
$$ \frac{1-x}{x}\cdot \sum_{i=1}^n \mathbb{P}\Big(\{\widetilde{X}_i=x\}\cap \widetilde{A}\Big) \ = \ \sum_{i=1}^n \mathbb{P}\Big(\{\widetilde{X}_i=1-x\}\cap \widetilde{A}\Big). $$

(iv)  We have the equality
$$ \mathbb{P}\Big(\max_{1\le i < j\le n}|X_i-X_j|\ge \delta\Big) \ = 2\cdot \mathbb{P}\Big( \Big\{\max_{1\le i < j\le n}|\widetilde{X}_i-\widetilde{X}_j|\ge \delta\Big\} \cap \widetilde{A}  \Big). $$
\end{prop}
\begin{proof} Since $U$ is measurable with respect to $\widetilde{\mathcal{G}}_i$ and
independent of $A$, we obtain
$$\mathbb{E}\big(\mathbbm{1}_{\widetilde{A}}|\mathcal{\widetilde{G}}_i\big) = \mathbbm{1}_{\{U=1\}}\mathbb{E}\big(\mathbbm{1}_A | \mathcal{\widetilde{G}}_i\big) + \mathbbm{1}_{\{U=0\}}\mathbb{E}\big(\mathbbm{1}_{A^c} | \mathcal{\widetilde{G}}_i\big) = UX_i+(1-U)(1-X_i)$$
and
$$\mathbb{P}(\widetilde{A})=\mathbb{E}\Big[\mathbb{E}\big(\mathbbm{1}_{\widetilde{A}}|\mathcal{\widetilde{G}}_i\big)\Big]=\mathbb{E}\big[UX_i+(1-U)(1-X_i)\big]=\frac{1}{2}$$
for all $i$. It remains to note that since $U\in \{0,1\}$, the set of all values  attained by $\widetilde{X}_i$ has at most $2m$ elements; this gives  (i). To show (ii), observe that 
\begin{align*}
\mathbb{P}\Big(\bigcap_{i=1}^{n}\{\widetilde{X}_i=x_i\} \cap \widetilde{A} \Big)
&=\mathbb{P}\Big(\bigcap_{i=1}^{n}\{\widetilde{X}_i=x_i\} \cap \widetilde{A}\cap \{U=0\} \Big)+
\mathbb{P}\Big(\bigcap_{i=1}^{n}\{\widetilde{X}_i=x_i\} \cap \widetilde{A} \cap \{U=1\}\Big)\\
&=\mathbb{P}\Big(\bigcap_{i=1}^{n}\{X_i=1-x_i\} \cap {A}^c\cap \{U=0\} \Big)+
\mathbb{P}\Big(\bigcap_{i=1}^{n}\{{X}_i=x_i\} \cap A \cap \{U=1\}\Big).
\end{align*}
Since $U$ is independent of $X_i$'s and $A$, and satisfies $\mathbb{P}(U=0)=\mathbb{P}(U=1)=1/2$, the above expression is equal to
\begin{align*}
&\mathbb{P}\Big(\bigcap_{i=1}^{n}\{X_i=1-x_i\} \cap {A}^c\cap \{U=1\} \Big)+
\mathbb{P}\Big(\bigcap_{i=1}^{n}\{{X}_i=x_i\} \cap A \cap \{U=0\}\Big)\\
&=\mathbb{P}\Big(\bigcap_{i=1}^{n}\{\widetilde{X}_i=1-x_i\} \cap \widetilde{A}^c\cap \{U=1\} \Big)+
\mathbb{P}\Big(\bigcap_{i=1}^{n}\{\widetilde{X}_i=1-x_i\} \cap \widetilde{A}^c \cap \{U=0\}\Big)\\
&=\mathbb{P}\Big(\bigcap_{i=1}^{n}\{\widetilde{X}_i=1-x_i\} \cap \widetilde{A}^c\Big),
\end{align*}
so (ii) is established. To prove the third part, fix  $x\in (0,1]$ and write
\begin{align*} \frac{1-x}{x}\cdot \sum_{i=1}^n \mathbb{P}\Big(\{\widetilde{X}_i=x\}\cap \widetilde{A}\Big)  \ \ &= \ \ \sum_{i=1}^n \mathbb{P}\Big(\{\widetilde{X}_i=x\}\cap \widetilde{A}^c\Big)= \  \sum_{i=1}^n \mathbb{P}\Big(\{\widetilde{X}_i=1-x\}\cap \widetilde{A}\Big),\end{align*}
where the first equality is due to the Lemma \ref{lemat} and the second is a consequence of (ii).  Finally, fix $m \in \{1, 2,\dots\}$, $X \in \mathcal{C}(n,m)$ and notice that 
$$\max_{1\le i<j\le n}|X_i-X_j| \ = \ \max_{1\le i < j \le n} |(1-X_i)-(1-X_j)|$$
almost surely. Hence we deduce (iv)  from
\begin{align*}\mathbb{P}\Big(\max_{1\le i < j\le n}|X_i-X_j|\ge \delta\Big)  \ \ &= \ \ \mathbb{P}\Big(\max_{1\le i < j\le n}|\widetilde{X}_i-\widetilde{X}_j|\ge \delta\Big) = \ 2\cdot \mathbb{P}\Big( \Big\{\max_{1\le i < j\le n}|\widetilde{X}_i-\widetilde{X}_j|\ge \delta\Big\} \cap \widetilde{A}  \Big),\end{align*}
as desired.
\end{proof}

As a direct consequence, we have the following crucial reduction.

\begin{cor}\label{reduction}
We have the inequality
\begin{equation}\label{red-iden}
 \sup_{\substack{m\in \{1,2,\dots\} \\ X \in\mathcal{C}(n,m)}} \mathbb{P}\Big(\max_{1\le i < j\le n}|X_i-X_j|\ge \delta\Big) \ \le  \  2\cdot \sup_{\substack{m  \in \{1,2,\dots\} \\ X \in\mathcal{C}'(n,m)}} \mathbb{P}\Big( \Big\{\max_{1\le i < j\le n}|X_i-X_j|\ge \delta\Big\} \cap A  \Big), 
\end{equation}
where $\mathcal{C}'(n,m)$ is the subset of all those $X \in C(n,m)$ that satisfy $\mathbb{P}(A)=\frac{1}{2}$ and 
\begin{equation}\label{A-Ac}
 \frac{1-x}{x}\cdot \sum_{i=1}^n \mathbb{P}\Big(\{{X}_i=x\}\cap {A}\Big) \ = \ \sum_{i=1}^n \mathbb{P}\Big(\{ {X}_i=1-x\}\cap  {A}\Big)\qquad \mbox{ for all }x\in (0,1]. 
\end{equation}
\end{cor}
\begin{proof}
By Proposition \ref{SYM7} (iii) and (iv), the left-hand side of \eqref{red-iden} does not exceed the right-hand side. 
% To show the reverse bound, it suffices to note that $C'(n,m)\subset C(n,m)$ and the condition $X\in C'(n,m)$ implies
%$$ 2\mathbb{P}\Big( \Big\{\max_{1\le i < j\le n}|X_i-X_j|\ge \delta\Big\} \cap A  \Big)=\mathbb{P}\Big(\max_{1\le i < j\le n}|X_i-X_j|\ge \delta  \Big).\qquad \qedhere$$
\end{proof}

It will later become clear that (\ref{red-iden}) is in fact an equality. As for now, the above argumentation allows us to reduce our main problem to the identification of 
\begin{equation} \label{sup-new}\sup_{\substack{m  \in \{1,2,\dots\} \\ X \in\mathcal{C}'(n,m)}} \mathbb{P}\Big( \Big\{\max_{1\le i < j\le n}|X_i-X_j|\ge \delta\Big\} \cap A  \Big).\end{equation}
The advantage over the original formulation (appearing on the left-hand side of \eqref{red-iden}) lies in the fact that we study the behavior of $X$ restricted to the set $A$. As we will see, the analysis of this expression can be performed in a purely analytic setup, with the use of combinatorial arguments. Consider the measure space $(\mathbb{R}_+,\mathcal{B}(\mathbb{R}_+),\lambda)$, where $\lambda$ stands for the Lebesgue measure. For $k\in\{1,2,\dots\}$, denote by $\Lambda(k)$ the family of all those functions $(H,L):\mathbb{R}_+ \to [0,1]^2$, which satisfy the following four requirements:
\begin{enumerate}
\item $L(x) \le \frac{1}{2} \le H(x)$ for all $x\in\mathbb{R}_+$, 
\item $H$ and $L$ are right-continuous step functions with a finite number of steps,
\item $\lambda(H>\frac{1}{2}) + \lambda(L<\frac{1}{2}) \le  \frac{k}{2},$
\item for any $y\in (0,1]$ we have
$$\frac{1-y}{y}\cdot \Big(\lambda(H=y) + \lambda(L=y) \Big) \ = \ \lambda(H=1-y) + \lambda(L=1-y).$$
\end{enumerate}

Here is a key statement, which links the above probabilistic considerations with the analytic context we have just introduced.

\begin{prop} \label{LAMB} The value of (\ref{sup-new}) is not bigger than 
\begin{equation} \label{sup-new-HL}  \sup_{(H,L) \in \Lambda(n)} \lambda\big( H\ge L+\delta   \big). \end{equation} \end{prop}

\begin{proof} Fix $m\in \{1, 2, \dots\}$, $X\in \mathcal{C}'(n,m)$ and the corresponding event $A$. We will construct $(H_X,L_X)\in \Lambda(n)$ such that 
\begin{equation} \label{P<l} \mathbb{P}\Big( \Big\{\max_{1\le i < j\le n}|X_i-X_j|\ge \delta\Big\} \cap A  \Big) \ = \ \lambda(H_X \ge L_X+\delta). \end{equation}
As $X\in \mathcal{C}'(n,m)$, there exists a natural number $l\le m^n$ such that $X$ takes  exactly $l$ different values. It follows that $A$ can be partitioned into disjoint family $\{A_k\}_{k=1}^l$ of events of positive probability, so that $X$ is constant on every element of this partition; let $(X_1,\dots, X_n)\equiv(x_1^{(k)},x_2^{(k)},\dots, x_n^{(k)})$ on $A_k$. For $1\le k \le l$, we set

$$p_k \ \ = \ \ \left\{
  \begin{array}{@{}ll@{}}
    n\mathbb{P}(A_k) & \text{if} \ \ \ \ \max_{1\le i < j\le n}|X_i-X_j|< \delta \ \ \ \text{on} \ \ A_k,\\
    (n-1)\mathbb{P}(A_k) & \text{if} \ \ \ \ \max_{1\le i < j\le n}|X_i-X_j|\ge \delta \ \ \ \text{on} \ \ A_k,
  \end{array}\right.$$
and introduce a disjoint partition $\mathcal{I}=\{I_k\}_{k=1}^{l+1}$ of $\mathbb{R}_+$ by
$$I_k \ \ = \ \ \left\{
  \begin{array}{@{}ll@{}}
    $[$p_1+\dots+p_{k-1}, p_1+\dots+p_{k}) & \text{for} \ \ 1\le k \le l, \\
     $[$p_1+\dots+p_l,\infty) & \text{for} \ \  k=l+1.
  \end{array}\right.$$
Now we are ready to define $(H_X,L_X)$, setting its values on each element of $\mathcal{I}$ separately. Assume that $k\in\{1,\,2,\,\ldots,\,l+1\}$ and distinguish three major cases.

\smallskip

$\cdot$ If $k=l+1$, we put $H_X(x)=L_X(x)=\frac{1}{2}$ for all $x\in I_{l+1}$.

\smallskip

$\cdot$ If $k\leq l$ and $p_k=n\mathbb{P}(A_k)$, we split $I_k$ into $n$ consecutive intervals $\{I_{k,s}\}_{s=1}^n$ (left-closed, right-open) of equal length and set
   \begin{equation}\label{defHL}
 H_X(x)= \max\Big(x_{s}^{(k)}, \frac{1}{2}\Big),\qquad \qquad  L_X(x)= \min \Big(x_{s}^{(k)},\frac{1}{2} \Big) 
\end{equation}
whenever $x\in I_{k,s}$ and $s=1,\,2,\,\ldots,\,n$.

\smallskip
 
$\cdot$ Finally, suppose that $k\leq l$ and $p_k=(n-1)\mathbb{P}(A_k)$. Then there are two indices $1\leq i_1<i_2\leq n$  such that $X_{i_1}\ge X_{i_2}+\delta$ on $A_k$ (the choice of $i_1,\,i_2$ may not be unique; in such a case, we pick any pair with this property). We divide $I_k$ into $n-1$ consecutive intervals $I_{k,s}$ of equal length, $s\in \{1,\,2,\,\dots,\,n\}\setminus \{i_2\}$, and put
$$ H_X(x)=x_{i_1}^{(k)},\qquad L_X(x)=x_{i_2}^{(k)}\qquad \mbox{if }x\in I_{k,i_1},$$
while for $x\in I_{k,s}$ and $s\in\{1,\,2,\,\ldots,\,n\}\setminus \{i_1,i_2\}$, we use \eqref{defHL}. In other words, we proceed as in the previous case, but the intervals $I_{k,i_1}$ and $I_{k,i_2}$ are now ``glued'' into one.

\smallskip

Let us check that the function $(H,L)$ we have just obtained does belong to $\Lambda(n)$, i.e., it satisfies the four requirements 1.-4..   The first two conditions hold directly by the construction. To verify the point 3., we inspect carefully the three cases considered above. Note that $H=L=1/2$ on $I_{l+1}$, so 
$$ \lambda(\{H>1/2\}\cap I_{l+1})+\lambda(\{L<1/2\}\cap I_{l+1})=0.$$
 If $k\leq l$ and $p_k=n\mathbb{P}(A_k)$, then the restrictions of $H$ and $L$ to $I_k$ are given by \eqref{defHL}; directly by this formula, we see that the sets $\{H>1/2\}\cap I_k$ and $\{L<1/2\}\cap I_k$ are disjoint and hence 
$$ \lambda(\{H>1/2\}\cap I_k)+\lambda(\{L<1/2\}\cap I_k)\leq \lambda(I_k)=p_k=n\mathbb{P}(A_k).$$
Finally, if $k\leq l$ and $p_k=(n-1)\mathbb{P}(A_k)$, then the above construction implies that the intersection of $\{H>1/2\}\cap I_k$ and $\{L<1/2\}\cap I_k$ is precisely the interval $I_{k,i_1}$. Consequently,
$$ \lambda(\{H>1/2\}\cap I_k)+\lambda(\{L<1/2\}\cap I_k)\leq \lambda(I_{k,i_1})+\lambda(I_k)=n\mathbb{P}(A_k).$$
Summing the above inequalities/equalities over $k$ and noting that $\mathbb{P}(A_1)+\mathbb{P}(A_2)+\ldots+\mathbb{P}(A_l)=\mathbb{P}(A)=1/2$, we obtain 3. It remains to note that the last property is a direct consequence of (\ref{A-Ac}).
\end{proof}

The next step is the following reduction.
  
\begin{prop} \label{LAMBdelta} The quantity (\ref{sup-new-HL}) can be rewritten as
  \begin{equation} \label{sup-new-HL*}  
\sup_{(H,L) \in \Lambda^{\delta}(n)} \lambda( H\ge L+\delta ),
\end{equation}
  where $\Lambda^{\delta}(n)$ is the subset of all $(H,L)\in \Lambda(n)$ satisfying
  \begin{equation}\label{star1}  \Big \{H\in \Big(\frac{1}{2}, \delta\Big) \Big\} \cup \Big \{L\in \Big(1-\delta, \frac{1}{2} \Big) \Big\} = \emptyset \end{equation} 
  and
  \begin{equation}\label{star2} 
\{ H\ge L +\delta \} = \Big\{ L<\frac{1}{2} \Big\}. 
\end{equation} 
\end{prop}
\begin{proof} Fix $(H,L)\in \Lambda(n)$ and assume that the condition (\ref{star1}) or (\ref{star2}) is not satisfied.  If \eqref{star1} fails, we modify $H$ and/or $L$ on the ``bad'' sets, changing their values to $\frac{1}{2}$ there. After this modification, the points $1$-$4$. are still satisfied and the value of $\lambda(H\ge L+\delta)$ remains unchanged. Now suppose that (\ref{star2}) does not hold. Because of the trivial inclusion $\{ H\ge L +\delta \} \subseteq \{ L<1/2\}$ and the equality $\{ L<1/2\}=\{ L\le1-\delta\}$ we have just guaranteed, there must exist $0\le a <b$ and $0<\gamma \le 1-\delta$ such that $L=\gamma$ and $H<\gamma+\delta$ on $[a,b)$. By point $4$., we can find pairwise disjoint intervals  $[a_j,b_j)$, $j=1,\dots,m$, satisfying
$$\bigcup_{j=1}^m [a_j,b_j)   \subset  \{H=1-\gamma\}\qquad \mbox{ and }\qquad \sum_{j=1}^m (b_j-a_j)    =   \frac{1-\gamma}{\gamma}\cdot(b-a).$$ 
Therefore, we can perform the following rearrangement:
\begin{enumerate}
\item change $L$ on $[a,b)$ from $\gamma$ to $\frac{1}{2}$,
\item change $H$ on $\bigcup_{j=1}^m [a_j,b_j)$ from $1-\gamma$ to $1$.
\end{enumerate}
This ``corrects'' the behavior  of $(H,L)$ on the troublesome interval $[a,b)$. Note that the obtained function belongs to $\Lambda(n)$ and the  value of $\lambda(H\ge L+\delta)$ is not decreased. It remains to observe
that we may guarantee the validity of (\ref{star2}), by performing sufficiently many such transformations. 
\end{proof}

The central part of the proof is the following estimate.%; note that
 %$\lambda(H>\frac{1}{2})  >  \lambda (L<\frac{1}{2})$
 % for every $(H,L)\in \Lambda^{\delta}(k)$ and $k=1,2,3,\dots$

\begin{lemma} \label{centre} We have the identity 
$$\phi:= \sup  \frac{\lambda(L<\frac{1}{2})}{\lambda(H>\frac{1}{2})-\lambda(L<\frac{1}{2})} \ = \ \frac{1-\delta}{\delta},$$
where the supremum is taken over all $k\in \{1,\,2,\,\ldots\}$ and all $(H,L)\in \Lambda^{\delta}(k)$ satisfying $\lambda(H>1/2)>0$.
\end{lemma}
We postpone the proof of this lemma to the next section, and proceed with our main result. 
\begin{proof}[Proof of Theorem \ref{contrN}] By Propositions \ref{DISCR}, \ref{SYM7} (iv), \ref{LAMB} and \ref{LAMBdelta}, we can write
\begin{align*} \sup_{(X_1,X_2,\dots, X_n)\in \mathcal{C}} \mathbb{P}\Big(\max_{1\le i < j\le n}|X_i-X_j|\ge \delta\Big)  \ \ &\le \ \ 2\cdot \sup_{(H,L) \in \Lambda(n)} \lambda(H\ge L+\delta)\\
&= \  2\cdot \sup_{(H,L) \in \Lambda^{\delta}(n)} \lambda \Big(L<\frac{1}{2}\Big).\end{align*}
Fix $(H,L)\in \Lambda^{\delta}(n)$. By Lemma \ref{centre}, we have
$  \lambda (L<1/2) \le (1-\delta)\lambda (H>1/2)$,
while the point $3$. gives 
$
 \lambda (L<1/2) + \lambda (H>1/2)  \le n/2.$ 
Combining these two estimates, we immediately obtain
$$ \mathbb{P}\Big(\max_{1\le i < j\le n}|X_i-X_j|\ge \delta\Big) \le  \frac{n(1-\delta)}{2-\delta}.$$
It remains to prove the sharpness of (\ref{EqN}). Observe that the function $\delta\mapsto (1-\delta)/(2-\delta)$ is decreasing on $[0,1]$, so the claim will follow if we construct an appropriate coherent vector $(Z_i)_{i=1}^n$ for every $\delta$ with $n(1-\delta)/(2-\delta)\le 1$. To this end, let
$\{A_0, A_1, \dots, A_n\} \cup \{B_0, B_1,\dots, B_n\}$ be a measurable partition of $\Omega$ satisfying
$$ \mathbb{P}(A_0)=\mathbb{P}(B_0)=\frac{1}{2}\cdot\left(1-\frac{n(1-\delta)}{2-\delta}\right)$$
and
$$\mathbb{P}(A_i)=\mathbb{P}(B_i)= \frac{1}{2}\cdot \frac{1-\delta}{2-\delta}\quad \mbox{   for  } 1\le i \le n.$$
Put $A=\bigcup_{i=0}^{n}A_i$, $B=\bigcup_{i=0}^n B_i$ and consider the $\sigma$-algebras
$$\mathcal{F}_i=\sigma \Big(A_i, B_i, (A\cup B_{i+1})\setminus A_{i+1}, (B\cup A_{i+1})\setminus  B_{i+1}  \Big),\qquad i=1,\,2,\,\ldots,\,n$$
(with the cyclic convention $A_{n+1}=A_1$, $B_{n+1}=B_1$). It is straightforward to check that the variables $Z_i=\mathbb{E}(\mathbbm{1}_{A}|\mathcal{F}_i)$, $i=1,\,2,\,\ldots,\,n$, satisfy
 $$ Z_i =\begin{cases}
1 & \mbox{on } \ A_i,\smallskip\\
\displaystyle 0 & \mbox{on } \ B_i,\\
\displaystyle \delta & \mbox{on } \ (A\cup B_{i+1})\setminus (A_i\cup A_{i+1}),\\
\displaystyle 1-\delta & \mbox{on } \ (B\cup A_{i+1})\setminus (B_i\cup B_{i+1}).
\end{cases}$$
Consequently, we have $\max_{1\le i < j\le n}|Z_i-Z_j|\ge \delta$ on each $A_k$ and each $B_k$; this proves the estimate 
$$\mathbb{P}\left(\max_{1\le i < j\le n}|Z_i-Z_j|\ge \delta\right)\geq \mathbb{P}(\Omega\setminus (A_0\cup B_0))=\frac{n(1-\delta)}{2-\delta},$$
which is the desired lower bound.
\end{proof}

\section{Proof of Lemma \ref{centre}} 

We will use some basic terminology from the theory of graphs. Recall that a simple (directed) graph $G$ is an ordered pair  $(V_G,E_G)$, where $V_G$ is the set of vertices and $E_G\subset V_G\times V_G$ is the collection of all edges. A simple graph is called a tree, if any two vertices are connected by exactly one path; a forest is a disjoint union of trees. 

From now on, we will use a shorter notation and write $\Lambda^{\delta}(\mathbb{N})$ instead of $\bigcup_{k=1}^{\infty}\Lambda^{\delta}(k)$. We start with an arbitrary $(H,L)\in \Lambda^{\delta}(\mathbb{N})$ satisfying $\lambda(H>1/2)>0$ and describe how such a function gives rise to a (directed) forest graph $\mathcal{T}_{(H,L)} = (\mathcal{V}_{(H,L)}, \mathcal{E}_{(H,L)})$. We will proceed  by induction, the intervals under consideration will always be left-closed and right-open:

\begin{enumerate}
\item Induction base. By 4., we have $\lambda(L<1/2)<\lambda(H>1/2)$ and hence $\lambda(H>\frac{1}{2}, L=\frac{1}{2})>0$. Therefore, we can find a finite family $\mathcal{V}_1=\{I_1^{(1)}, I_2^{(1)},\dots, I_{k_1}^{(1)}\}$ of disjoint intervals, such that
$$\bigcup_{j=1}^{k_1} I_j^{(1)} \ = \ \Big\{H>\frac{1}{2},\,L=\frac{1}{2}\Big\}$$
and such that $H$ is constant on each interval, say, $H=x_j^{(1)}$ on $I_j^{(1)}$ for $1\le j\le k_{1}$. Set $\mathcal{E}_1=\emptyset$.

\item Induction step. Suppose that we have successfully constructed $\mathcal{V}_j$ and $\mathcal{E}_j$ for $j\le i-1$. Moreover, assume that $\mathcal{V}_{i-1}=\{I_{1}^{(i-1)}, I_2^{(i-1)},\dots, I_{k_{i-1}}^{(i-1)}\}$ and $H=x_{j}^{(i-1)}$ on $I_{j}^{(i-1)}$ for $1\le j \le k_{i-1}$. By point $4$. there exists a finite family $\bigcup_{j=1}^{k_{i-1}}\{J_1^j, J_2^j, \dots, J_{m_j}^j\}$ of disjoint intervals, such that
$$\bigcup_{l=1}^{m_j} J_{l}^j  \ \ \subset \ \  \{L=1-x_j^{(i-1)}\}\setminus \bigcup_{n=1}^{i-1} \bigcup \mathcal{V}_n,\quad \qquad \sum_{l=1}^{m_j} \lambda(J_{l}^j)    =   \frac{1-x_j^{(i-1)}}{x_j^{(i-1)}}\cdot \lambda(I_{j}^{(i-1)})$$
for $j=1,2,\dots, k_{i-1}$, and such that $H$ is constant on each $J_l^j$. Set
$$\mathcal{V}_i \ = \ \bigcup_{j=1}^{k_{i-1}}\{J_1^j, J_2^j, \dots, J_{m_j}^j\}\qquad \mbox{ and }\qquad 
\mathcal{E}_i \ = \ \mathcal{E}_{i-1}\cup \bigcup_{j=1}^{k_{i-1}}\{I_{j}^{(i-1)}\} \times \{J_1^j, J_2^j, \dots, J_{m_j}^j\}, $$
\end{enumerate}
and put $\mathcal{V}_{(H,L)}=\bigcup_{i=1}^{\infty}\mathcal{V}_i$, $\mathcal{E}_{(H,L)}=\bigcup_{i=1}^{\infty}\mathcal{E}_i$. 

\smallskip

To gain some intuition about the above construction, it is convenient to carry out an explicit calculation. 

\begin{exe} Let $\delta=0.7$ and consider a pair $(H,L)$ given by
\begin{align*}
 H&=\chi_{[1,3)}+\frac{7}{8}\big(\chi_{[0,1)}+\chi_{[3,5)}+\chi_{[8,12)}\big)+\frac{3}{4}\big(\chi_{[5,8)}+\chi_{[12,15)}\big)+\frac{1}{2}\chi_{[15,\infty)},\\
 L&=\frac{1}{8}\chi_{[0,1)}+\frac{1}{4}\chi_{[1,3)}+\frac{1}{2}\chi_{[3,\infty)}.
 \end{align*}
\begin{figure}[htbp]
\begin{center}\includegraphics[scale=0.7]{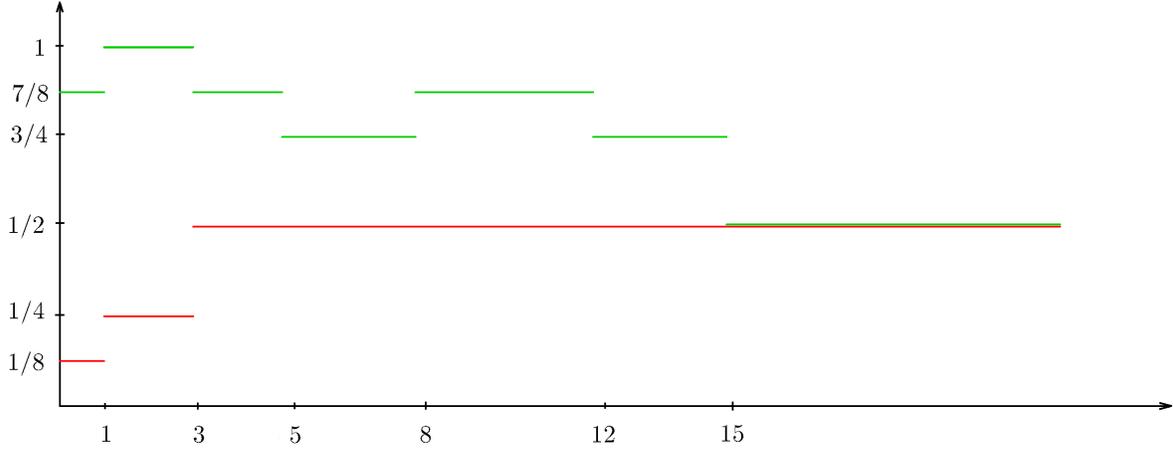}
\caption{
The graphs of the function $H$ (green) and $L$ (red).}\label{graph}
\end{center}
\end{figure} 
It is not difficult to check that $(H,L)\in \Lambda^\delta(36)$. Let us now explain the construction of the forest $\mathcal{T}_{(H,L)}$. The starting point is to look at the set $\{H>1/2,L=1/2\}=[3,15)$. In our case, this set splits into four intervals on which $H$ is constant: $[3,5)$, $[5,8)$, $[8,12)$ and $[12,15)$. These four intervals are the roots of four trees which will form the forest $\mathcal{T}_{(H,L)}$. Next, for each root we describe its descendants; it is best to explain the procedure on a given root, say, $[5,8)$. The length of the interval is equal to $3$ and the function $H$ is equal to $3/4$ there. The application of the property 4. with $y=3/4$ gives 
\begin{equation}\label{constraint}
\lambda(H=3/4)=3\lambda(L=1/4),
\end{equation}
i.e., the set $\{L=1/4\}$ is three times smaller than $\{H=3/4\}$. The children of $[5,8)$ are the pairwise disjoint subintervals $J^1_1,J^1_2,\ldots,J^1_{m_1}$ of $\{L=1/4\}$ for which the measure constraint \eqref{constraint} is preserved:
$$ \lambda([5,8))=3\lambda\Big(\bigcup_{j=1}^{m_1} J^1_j\Big),$$
and such that $H$ is constant on each $J^1_j$. There is a lot of ambiguity with the choice of $J^j$'s, we may actually take a single child $J^1_1=[1,2)$. We carry out a similar procedure with each root, making sure that all the children obtained in the process are pairwise disjoint. For example, at the end we may obtain the following (partial) forest:

\smallskip

\begin{figure}[htbp]
\begin{center}\includegraphics[scale=1.1]{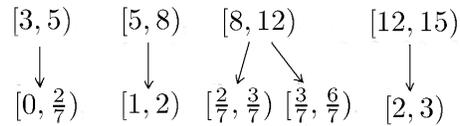}
\caption{Partial forest: roots and their children. 
}\label{graph2}
\end{center}
\end{figure}

Note that the intervals $[\frac{2}{7},\frac{3}{7})$ and $[\frac{3}{7},\frac{6}{7})$ could as well be merged into one $[\frac{2}{7},\frac{6}{7})$: then the root $[8,12)$ would have just one descendant. Next, we continue the procedure, but now the role of the roots is played by the children of the first generation which have been just constructed. It is clear that the procedure is well-defined: by property 4.,  at each step there are no problems with the existence of intervals satisfying appropriate measure and disjointness requirements.  We would just like to mention  that any interval on which $H$ is equal to $1$ does not have any descendants (the tree is cut at such a vertex).
\end{exe}

%There is a natural ordering in the set of vertices. Namely, if $(I,J)\in \mathcal{E}_{(H,L)}$ is an arbitrary edge, we will say that  $J$ is generated by $I$ (or that $I$ generates $J$). Note that %this relation is not symmetric: if $J$ is generated by $I$ then $I$  can not be generated by $J$. Moreover, 
%every interval $J\in \mathcal{V}_{(H,L)}\setminus \mathcal{V}_1$ is generated by exactly one vertex (that is, $J$ has exactly one ``father``); furthermore, the intervals $I_1^{(1)}, I_2^{(1)},\dots, I_{k_1}^{(1)}$, being the roots of the forest $\mathcal{T}_{(H,L)}$,  are not generated by anything. 

Now we will establish the following important fact.

\begin{prop}The family $\mathcal{V}_{(H,L)}$ is disjoint and $$\bigcup \mathcal{V}_{(H,L)} \ = \ \Big\{H>\frac{1}{2}\Big\},$$ up to a set of measure zero.
\end{prop}
\begin{proof} The first part follows from the very construction. To prove the second part, we  will first show inductively that 
\begin{equation}\label{ind2}
H>\frac{1}{2}\quad \mbox{ for every }J\in \mathcal{V}_{(H,L)}.
\end{equation}
Indeed, we obviously have $H>\frac{1}{2}$ on $\bigcup \mathcal{V}_1$. So, fix $i\in \{2,3,\dots\}$ and assume that $H>\frac{1}{2}$ on $\bigcup_{n=1}^{i-1}\bigcup \mathcal{V}_n$. Let $J\in \mathcal{V}_{i}$ be an arbitrary interval and let $I\in \mathcal{V}_{i-1}$ be the father of $J$ (relative to the structure of the tree $\mathcal{T}_{(H,L)}$). Then there exists $x>\frac{1}{2}$ such that $H\equiv x$ on $I$ and $L\equiv 1-x<\frac{1}{2}$ on $J$. By the definition of $\Lambda^{\delta}(\mathbb{N})$, we have $ \big\{ L<1/2 \big\} = \{ H\ge L +\delta \}$
and hence $H\ge 1-x+\delta \ge \delta> \frac{1}{2}$ on $J$. This completes the proof of \eqref{ind2}. 
To show the reverse inclusion (up to a set of measure zero), put $U=\big\{H>1/2 \big\} \setminus \bigcup \mathcal{V}_{(H,L)}$ 
and assume that $\lambda(U)>0$. Recall that, again by the definition of $\Lambda^{\delta}(\mathbb{N})$, we have
$$\Big\{H\in \Big(\frac{1}{2},\delta\Big)\Big\} = \Big\{ L\in \Big(1-\delta, \frac{1}{2}\Big) \Big\} = \emptyset.$$
Fix $y\in[\delta,1]$ and note that by the construction of the sets $\mathcal{V}_i$ above,  we may write
$$\frac{1-y}{y}\cdot \sum_{n=1}^{i-1} \ \sum_{\{I\in \mathcal{V}_n: \  H\equiv y \  \text{on} \   I\}} \lambda(I) \ \ = \ \ \sum_{n=2}^{i} \ \sum_{\{J\in \mathcal{V}_n: \  L \equiv 1-y \  \text{on} \   J\}} \lambda(J),$$
for all $i=2,3,\dots$. Hence, passing with $i$ to infinity yields
\begin{equation} \label{4.Vhl}  \frac{1-y}{y}\cdot \sum_{\{I\in \mathcal{V}_{(H,L)}: \  H\equiv y \  \text{on} \   I\}} \lambda(I) \ \ = \ \  \ \sum_{\{J\in \mathcal{V}_{(H,L)}: \  L\equiv 1-y \  \text{on} \   J\}} \lambda(J).   \end{equation} 
On the other hand, just by the property $4$., we have
\begin{equation} \label{4.(H,L)} \frac{1-y}{y}\cdot \lambda(H=y) \ = \ \lambda(L=1-y).\end{equation}
Subtracting (\ref{4.Vhl}) from (\ref{4.(H,L)}), we get
\begin{equation} \label{4.U} \frac{1-y}{y}\cdot \lambda \Big(\{H=y\}\cap U \Big) \ = \ \lambda \Big(\{L=1-y\}\cap U \Big),\end{equation}
for $y\in[\delta,1]$. Next, by the property $2$., there exists a finite sequence  $y_1, y_2,\dots, y_k \in [\delta,1]$, satisfying
$$\lambda \Big(\Big\{H\not \in \{y_1,y_2,\dots, y_k\}\Big\}\cap U \Big) \ = \ \lambda \Big(\Big\{L\not \in \{1-y_1, 1-y_2,\dots, 1-y_k\} \Big\} \cap U \Big) \ = \ 0.$$
Therefore, summing (\ref{4.U}) for $y_1,y_2,\dots, y_k$, we obtain the inequality
\begin{align*} \sum_{i=1}^k \lambda \Big(\{L=1-y_i\}\cap U \Big)  &= \sum_{i=1}^{k}\frac{1-y_i}{y_i}\cdot \lambda\Big(\{H=y_i\}\cap U\Big)\le \frac{1-\delta}{\delta}\cdot \sum_{i=1}^k \lambda\Big(\{H=y_i\} \cap U \Big),\end{align*}
and hence 
\begin{equation} \label{LU<HU} \lambda\Big(\Big\{L<\frac{1}{2}\Big\}\cap U\Big) \ < \ \lambda\Big(\Big\{H>\frac{1}{2}\Big\}\cap U\Big)\end{equation}
if only the right-hand side of (\ref{LU<HU}) is positive. At the same time, we have
$$U\cap \Big\{H>\frac{1}{2}, \ L=\frac{1}{2} \Big\} \ = \ \emptyset,$$
since the set $\{H>\frac{1}{2}, L=\frac{1}{2}\}$ has been already covered by $\mathcal{V}_1$. Consequently, we get
$$\Big\{H>\frac{1}{2}\Big\} \cap U \  = \ \Big\{L<\frac{1}{2} \Big\} \cap U \ =  \ U$$
and thus
$$ \lambda\Big(\Big\{L<\frac{1}{2}\Big\}\cap U\Big) \ = \ \lambda\Big(\Big\{H>\frac{1}{2}\Big\}\cap U\Big),$$
which contradicts (\ref{LU<HU}).
\end{proof}

We are ready to connect the above graph structure with the assertion of Lemma \ref{centre}. 
Under the notation we have just introduced, the expression for $\phi$ can be rewritten in the form

\begin{equation} \label{centre2} \phi =  \sup_{(H,L)\in \Lambda^{\delta}(\mathbb{N})} \frac{\sum_{J\in \mathcal{V}_{(H,L)\setminus \mathcal{V}_1}} \lambda(J)}{ \sum_{I\in \mathcal{V}_1} \lambda(I)}. \end{equation}
We split the forest $\mathcal{T}_{(H,L)}$ into the disjoint trees: for $1\le j \le k_1$, let $\mathcal{T}_{(H,L)}^{j}$ denote the directed tree with root $I_j^{(1)}$. Then we have
\begin{align*} \frac{\sum_{J\in \mathcal{V}_{(H,L)\setminus \mathcal{V}_1}} \lambda(J)}{ \sum_{I\in \mathcal{V}_1} \lambda(I)}  &= \frac{\sum_{j=1}^{k_1} \Big[\lambda \Big(\bigcup \mathcal{T}_{(H,L)}^j\Big)-\lambda \Big(I_j^{(1)}\Big)\Big]}{\sum_{j=1}^{k_1} \lambda(I_j^{(1)})} \le \max_{1\le j \le k_1} \frac{\lambda \Big(\bigcup \mathcal{T}_{(H,L)}^j\Big)-\lambda \Big(I_j^{(1)}\Big)}{ \lambda(I_j^{(1)})}. \end{align*}
This inequality leads to a convenient reduction: in the problem (\ref{centre2}) it is enough to consider $(H,L)$ with $\mathcal{T}_{(H,L)}=\mathcal{T}_{(H,L)}^1$, i.e. in the context when the underlying forest structure consists of a single tree. Let us discuss some further simplifications. With no loss of generality, we may assume that $\lambda(I_{1}^{(1)})=1$. Indeed, scaling $I_1^{(1)}$ by a factor $c>0$ results in scaling all intervals in $\mathcal{V}_2$ by the same factor, which, in turn, leads to the same scaling of all intervals generated by $\mathcal{V}_2$ (i.e. $\mathcal{V}_3$), and so on. Summarizing, we have obtained
\begin{equation} \label{centre3} \phi \ \ = \ \  \sup_{\Xi^{\delta}(\mathbb{N})} \Big[\lambda\Big(\bigcup \mathcal{V}_2\Big) +\lambda\Big(\bigcup \mathcal{V}_3\Big)+\dots\Big],  \end{equation}
where supremum is taken over
 $$\Xi^{\delta}(\mathbb{N}) := \ \Big\{(H,L)\in \Lambda^{\delta}(\mathbb{N}): \ \mathcal{T}_{(H,L)}=\mathcal{T}_{(H,L)}^1 \ \ \text{and} \ \ \lambda(I_{1}^{(1)})=1\Big\}.$$
 Note that the series under supremum in (\ref{centre3}) is uniformly convergent: by the construction, we have 
$$ \lambda \Big(\bigcup \mathcal{V}_{m+1}\Big) \ \le \ \frac{1-\delta}{\delta} \cdot \lambda\Big(\bigcup \mathcal{V}_{m}\Big),$$
for all $m=1,2,\dots$ Therefore, we can reformulate (\ref{centre3}) as
\begin{equation} \label{centre4} \phi \ = \ \ \lim_{m\to \infty } \ \Bigg( \sup_{\Xi^{\delta}(\mathbb{N})} \ \sum_{j=2}^m \ \lambda\Big(\bigcup \mathcal{V}_j\Big)\Bigg). \end{equation}

To compute the above supremum, it is convenient to apply dynamic programming techniques. Let  $\Phi: [\delta, 1]\to \mathbb{R}_+$  be given by 
\begin{equation}\label{BELL} 
\Phi(x) \  =  \ \ \sup_{\textbf{x}} \  \sum_{n=0}^{\infty} \  \prod_{i=0}^{n} \ \frac{1-x_i}{x_i},
\end{equation}
where the supremum is taken over all sequences $\textbf{x}=(x_0, x_1, x_2, \dots)$ satisfying
\begin{center} $x_0=x$, \ \ \ $x_n\in [\delta, 1]$ \  \ and \ \ $x_{n+1}\ge 1-x_n+\delta$ \ \ \ for  $n=0,1,2,\dots$ \end{center}
We may call $\Phi$ the Bellman function associated with \eqref{centre4}. Its connection to the problem is described in the following statement.

\begin{prop} We have the identity
 $$\phi  =  \sup_{x\in [\delta,1]}\Phi(x).$$
\end{prop}

\begin{proof} Fix $m\in \{1,2,\dots\}$. Analogously to the reduction  $\mathcal{T}_{(H,L)}=\mathcal{T}_{(H,L)}^1$, we easily verify that it is enough to handle $(H,L)$ with $\mathcal{V}_j = \{I_1^{(j)}\}$ for $1\le j \le m$.  Passing with $m$ to infinity, just as in (\ref{centre4}), we get
\begin{equation} \label{centre5} \phi \ \ = \ \  \sup_{\Xi^{\delta}(\mathbb{N})} \ \sum_{j=2}^{\infty} \ \lambda(I_{1}^{(j)}),   \end{equation}
where $I_{1}^{(n+1)}$ is generated by $I_{1}^{(n)}$ for each $n\ge 1$. Recall from construction that $H$ is constant on such intervals: denote $H=x_{n-1}$ on $I_1^{(n)}$, $n=1,2,\dots$ Note that inequality $x_{n+1}\ge 1-x_n+\delta$ is a straightforward consequence of $(H,L)\in \Lambda^{\delta}(\mathbb{N})$. Lastly, let $\Phi(x)$ denote the right-hand side of (\ref{centre5}) with an additional restriction to $x_0=x$. This yields the claim.
\end{proof}

We turn our attention to the identification of the formula for $\Phi$. We start with a structural property of the Bellman function.

\begin{prop}
 For any $x\in [\delta,1]$ we have the recurrence relation
\begin{equation}\label{BELL-re}\Phi(x) \ = \ \frac{1-x}{x}\Big(1+\sup_{y\ge 1-x+\delta} \Phi(y)  \Big).\end{equation}
\end{prop}
\begin{proof}
The argument rests on the so-called optimality principle. By (\ref{BELL}), we simply have
$$ \Phi(x) \  = \ \ \frac{1-x}{x}\cdot \Big(1 \ + \ \sup_{\widetilde{\textbf{x}}}\sum_{n=1}^{\infty} \ \prod_{i=1}^{n} \ \frac{1-x_i}{x_i} \Big),$$
 where supremum is taken over all sequences $\widetilde{\textbf{x}}=(x_1, x_2, \dots)$ such that 
 $$ x_1\ge 1-x+\delta,\qquad x_n\in [\delta, 1]\quad \mbox{ and }\quad x_{n+1}\ge 1-x_n+\delta\quad \mbox{  for  }n=1,2,3,\dots.  \qedhere $$
 \end{proof}

Now we will make use of the following procedure, which is often successful in the treatment of various problems in dynamic programming. Namely, based on some experimentation, we will \emph{guess} for which choice of $\textbf{x}$ the supremum defining $\Phi(x)$ is attained, thus obtaining ``a candidate'' $\Psi$ for the Bellman function. By the very definition, this candidate must satisfy $\Psi\leq \Phi$. The reverse estimate will be obtained by the verification that the candidate also satisfies the structural requirement \eqref{BELL-re}, and exploiting this condition appropriately.

 We proceed to the choice of $\textbf{x}$. A little thought and a closer inspection suggests that  problem  (\ref{BELL}) should be  maximized by an alternating sequence
 $$\hat{\textbf{x}} \ = \ (x,  1-x+\delta,  x,   1-x+\delta,   x, \ \dots).$$
 Indeed, this is quite a natural guess: we come up with $\hat{\textbf{x}}$ simply by assuming equalities in the contraints for the coordinates $x_0$, $x_1$, $x_2$, $\ldots$.  
 Plugging this sequence  into (\ref{BELL}), we compute the corresponding candidate for $\Phi(x)$, obtaining
 $$\Psi(x) := \ \  \frac{1-x}{x} \ + \ \frac{1-x}{x}\cdot \frac{x-\delta}{1-x+\delta} \ + \  \frac{1-x}{x}\cdot \frac{x-\delta}{1-x+\delta}\cdot \frac{1-x}{x} \ +\dots \ = \ \frac{1-x}{\delta},$$
 for all $x\in [\delta,1]$. Then $\Psi\leq \Phi$, as we have already commented above, so the proof will be complete if we manage to check that $\Psi \geq \Phi$.
 
\begin{proof}[Proof of Lemma \ref{centre}]  First, we show that $\Psi$ fulfills the recurrence (\ref{BELL-re}). Indeed, for $x\in [\delta,1]$, we have
 \begin{align*} \frac{1-x}{x}\Big(1+\sup_{y\ge 1-x+\delta} \Psi(y)  \Big)   &=   \frac{1-x}{x}\Big(1+ \Psi(1-x+\delta)  \Big) = \frac{1-x}{x}\Big(1+\frac{x-\delta}{\delta}  \Big) \ = \ \Psi(x). \end{align*}
 Pick any $x\in [\delta, 1]$ and $\varepsilon>0$. By (\ref{BELL}), we can choose an admissible  sequence $\textbf{x}_{\varepsilon}=(x_0, x_1, \dots)$ (i.e., satisfying $x_0=x$ and $x_{n+1}\ge 1-x_n+\delta$, $n=0,1,2,\dots$) such that
$$ \Phi(x) \  \le   \ \varepsilon \ + \ \sum_{n=0}^{\infty} \  \prod_{i=0}^{n} \ \frac{1-x_i}{x_i}. $$
 Since $\frac{1-x_i}{x_i} \le \frac{1-\delta}{\delta}$, $i=1,2,\dots$, there is a natural number $m$ for which
  \begin{equation} \label{APROX} \Phi(x) \  \le   \ 2\varepsilon \ + \ \sum_{n=0}^{m} \  \prod_{i=0}^{n} \ \frac{1-x_i}{x_i}. \end{equation}
 On the other hand, by recurrence relation (\ref{BELL-re}), we can write
 \begin{align*}
\Psi(x) = \Psi(x_0) \ge \frac{1-x_0}{x_0}\Big(1+ \Psi(x_1)  \Big)&=\frac{1-x_0}{x_0}+\frac{1-x_0}{x_0}\Psi(x_1)\\
  &\ge \frac{1-x_0}{x_0}+ \frac{1-x_0}{x_0} \frac{1-x_1}{x_1}\Big(1+ \Psi(x_2)  \Big)  \\
&= \frac{1-x_0}{x_0}+ \frac{1-x_0}{x_0} \frac{1-x_1}{x_1}+ \frac{1-x_0}{x_0} \frac{1-x_1}{x_1}\Psi(x_2)  
\end{align*}  
and so on. After $m$ steps, we obtain
$$ \Psi(x)\ge \sum_{n=0}^{m} \  \prod_{i=0}^{n} \ \frac{1-x_i}{x_i}+\left(\prod_{i=0}^{m} \ \frac{1-x_i}{x_i}\right)\Psi(x_{m+1})\ge \sum_{n=0}^{m} \  \prod_{i=0}^{n} \ \frac{1-x_i}{x_i}.$$
Hence, by (\ref{APROX}), we get $\Psi(x) + 2\varepsilon  \ \ge \ \Phi(x)$, and since  $\varepsilon > 0$ was chosen arbitrarily, the reverse bound $\Psi\geq \Phi$ follows. This proves the claim and completes the proof of \eqref{EqN}: $\phi  =  \sup_{x\in [\delta,1]}\Phi(x)=(1-\delta)/\delta.$
\end{proof}

%-----------------------------------------------------------------

\setlength{\baselineskip}{2ex}

\bibliographystyle{plain}
\bibliography{CPbib}

\begin{thebibliography}{10}

\bibitem{B4}
I.~Arieli and Y.~Babichenko.
\newblock A population's feasible posterior beliefs.
\newblock 2022.
\newblock (preprint) available at arXiv:2202.01846 [cs.GT].

\bibitem{B2}
I.~Arieli, Y.~Babichenko, and F.~Sandomirskiy.
\newblock Persuasion as transportation.
\newblock 2022.
\newblock (preprint) available at \url{https://fedors.info}.

\bibitem{B1}
I.~Arieli, Y.~Babichenko, F.~Sandomirskiy, and O.~Tamuz.
\newblock Feasible joint posterior beliefs.
\newblock {\em Journal of Political Economy}, 129, 2021.

\bibitem{contra}
K.~Burdzy and S.~Pal.
\newblock Can coherent predictions be contradictory?
\newblock {\em Advances in Applied Probability}, 53, 2021.

\bibitem{pitman}
K.~Burdzy and J.~Pitman.
\newblock Bounds on the probability of radically different opinions.
\newblock {\em Electron. Commun. Probab.}, 25, 2020.

\bibitem{mastersthesis}
S.~Cichomski.
\newblock Maximal spread of coherent distributions: a geometric and
  combinatorial perspective.
\newblock Master's thesis, University of Warsaw, 2020.
\newblock available at arXiv:2007.08022 [math.PR].

\bibitem{EJP}
S.~Cichomski and A.~Os\k ekowski.
\newblock The maximal difference among expert's opinions.
\newblock {\em Electronic Journal of Probability}, 26, 2021.

\bibitem{BPC}
S.~Cichomski and F.~Petrov.
\newblock A combinatorial proof of the burdzy-pitman conjecture.
\newblock 2022.
\newblock (preprint) available at arXiv:2204.07219 [math.CO].

\bibitem{C1}
A.~P. Dawid, M.~H. DeGroot, and J.~Mortera.
\newblock Coherent combination of experts' opinions.
\newblock {\em Test}, 4, 1995.

\bibitem{C2}
M.~H. DeGroot.
\newblock A bayesian view of assessing uncertainty and comparing expert
  opinion.
\newblock {\em Journal of Statistical Planning and Inference}, 20, 1988.

\bibitem{D2}
A.~Os\k ekowski.
\newblock {\em Sharp Martingale and Semimartingale Inequalities}.
\newblock Birkh{\"a}user Basel, 2012.

\bibitem{B3}
K.~He, F.~Sandomirskiy, and O.~Tamuz.
\newblock Private private information.
\newblock 2021.
\newblock (preprint) available at arXiv:2112.14356v2 [econ.TH].

\bibitem{D3}
I.~Pinelis, V.~H. de~la Pe\~{n}a, R.~Ibragimov, A.~Os\k ekowski, and
  I.~Shevtsova.
\newblock {\em Inequalities and Extremal Problems in Probability and
  Statistics: Selected Topics}.
\newblock Academic Press, 2017.

\bibitem{C4}
R.~Ranjan and T.~Gneiting.
\newblock Combining probability forecasts.
\newblock {\em Journal of the Royal Statistical Society: Series B (Statistical
  Methodology)}, 72, 2010.

\bibitem{C3}
V.~A. Satop{\"a}{\"a}, R.~Pemantle, and L.~H. Ungar.
\newblock Modeling probability forecasts via information diversity.
\newblock {\em Journal of the American Statistical Association}, 111, 2016.

\bibitem{D1}
V.~Vasyunin and A.~Volberg.
\newblock {\em The Bellman Function Technique in Harmonic Analysis}.
\newblock Cambridge University Press, 2020.

\end{thebibliography}

\end{document}